\documentclass[12pt]{article}

\usepackage[numbers]{natbib}
\usepackage{graphicx}
\usepackage{amssymb}
\usepackage{amsmath,amsthm,enumerate}

\usepackage[citecolor=blue,colorlinks=true]{hyperref} 

\usepackage[left=2cm,right=2cm,top=3.5cm,bottom=3.5cm]{geometry}
\usepackage{color}

\renewcommand{\footnoterule}{%
  \kern -3pt
  \hrule width \textwidth height 0.4pt
  \kern 2pt
}

\usepackage{extarrows}
\usepackage{stmaryrd}
\usepackage{import}
\usepackage{yfonts}
\usepackage{mathtools}
\usepackage{upgreek}
\usepackage{lmodern}
\usepackage{inputenc}
\usepackage{microtype}
\usepackage{float}
\usepackage{enumitem}
\setitemize{leftmargin=*}
\usepackage{tikz}
\usetikzlibrary{shapes,arrows}
\usepackage{chngcntr}
\usepackage[linesnumbered,german]{algorithm2e}
\usepackage{csquotes}
\usepackage{array}
\usepackage{multirow}
\usepackage{multicol}
\usepackage{cite}
\usepackage{accents}
\usepackage{marvosym}
\usepackage{wasysym}
\usepackage{bm}
\usepackage{dsfont}
\usepackage{xfrac}
\usepackage{hyperref}
\usepackage{xcolor}
\hypersetup{linktoc=all, colorlinks=true, linkcolor=blue, citecolor=blue}
\usepackage{stackengine}

\stackMath

\DeclareFontShape{OMX}{cmex}{m}{b}{<-> cmexb10}{}
\DeclareSymbolFont{boldlargesymbols}{OMX}{cmex}{m}{b}
\DeclareMathAccent{\bwidetilde}{\mathord}{boldlargesymbols}{"65}

\newtheoremstyle{break}
{14pt}{20pt}%
{}{}%
{\bfseries}{\vspace{0.5mm}}%
{\newline}{}%

\theoremstyle{break}

\theoremstyle{plain}
\newtheorem{theorem}{Theorem}[section]

\newtheorem{lemma}[theorem]{Lemma}

\theoremstyle{definition}
\newtheorem{definition}[theorem]{Definition}

\theoremstyle{remark}
\newtheorem{remark}[theorem]{Remark}

\catcode`\@=11 \@addtoreset{equation}{section}

\catcode`\@=12

\allowdisplaybreaks

\definecolor{Cgrey}{rgb}{0.85,0.85,0.85}
\definecolor{Cblue}{rgb}{0.50,0.85,0.85}
\definecolor{Cred}{rgb}{1,0,0}
\definecolor{fancy}{rgb}{0.10,0.85,0.10}

\newcommand\Cbox[2]{%
    \newbox\contentbox%
    \newbox\bkgdbox%
    \setbox\contentbox\hbox to \hsize{%
        \vtop{
            \kern\columnsep
            \hbox to \hsize{%
                \kern\columnsep%
                \advance\hsize by -2\columnsep%
                \setlength{\textwidth}{\hsize}%
                \vbox{
                    \parskip=\baselineskip
                    \parindent=0bp
                    #2
                }%
                \kern\columnsep%
            }%
            \kern\columnsep%
        }%
    }%
    \setbox\bkgdbox\vbox{
        \color{#1}
        \hrule width  \wd\contentbox %
               height \ht\contentbox %
               depth  \dp\contentbox
        \color{black}
    }%
    \wd\bkgdbox=0bp%
    \vbox{\hbox to \hsize{\box\bkgdbox\box\contentbox}}%
    \vskip\baselineskip%
}

\DeclareMathAlphabet{\mathbbm}{U}{bbm}{m}{n}
\newcommand{\softd}{d\hspace{-0.2mm}'}
\renewcommand{\S}{\mathbb{S}}
\newcommand{\Id}{\mathbb{I}}
\newcommand{\Lp}[2]{L^{\raisebox{#2}{\scalebox{0.75}{$#1$}}}}

\newcommand{\reals}{\mathbb{R}}
\newcommand{\fatu}{\bm{u}}

\newcommand{\fatx}{\bm{x}}

\newcommand{\bvarrho}{\bwidetilde{\varrho}}
\newcommand{\bfatu}{\bwidetilde{\fatu}}
\newcommand{\btheta}{\bwidetilde{\theta}}

\newcommand{\bS}{\bwidetilde{S}}

\newcommand{\const}{\mathrm{const}}

\newcommand{\p}{p\hspace{0.2mm}}

\newcommand{\po}{\partial \hspace{0.1mm} \Omega}

\newcommand{\thetazero}{\theta_{\hspace{-0.3mm}0}}

\newcommand{\inttau}{\int_{0}^{\hspace{0.3mm}\tau}}

\newcommand{\into}{\int_{\hspace{0.2mm}\raisebox{-0.3mm}{\scalebox{0.75}{$\Omega$}}}}

\newcommand{\inttauinto}{\inttau\!\!\into}
\newcommand{\dx}{\mathrm{d}\bm{x}}

\newcommand{\dt}{\mathrm{d}t}

\newcommand{\dxdt}{\dx\hspace{0.3mm}\dt}

\newcommand{\dd}{\mathrm{d}}

\newcommand{\pt}{\partial_t}
\newcommand{\gradx}{\nabla_{\hspace{-0.7mm}\fatx}}
\newcommand{\divx}{\mathrm{div}_{\fatx}}

\newcommand{\omean}[1]{\langle #1 \rangle}

\newcommand{\Ck}[1]{C^{\hspace{0.3mm}#1}}

\newcommand{\Cone}{C^{\hspace{0.3mm}1}}

\makeatletter
\renewenvironment{proof}[1][\proofname]{%
   \par\pushQED{\qed}\normalfont%
   \topsep6\p@\@plus6\p@\relax
   \trivlist\item[\hskip\labelsep\itshape\bfseries#1\@addpunct{.}]%
   \ignorespaces
}{%
   \popQED\endtrivlist\@endpefalse
}
\makeatother

\date{}

\begin{document}


\title{DMV-strong uniqueness principle for the compressible Navier-Stokes system with potential temperature transport\thanks{This work has been funded by the Deutsche Forschungsgemeinschaft (DFG, German Research Foundation) - Project number 233630050 - TRR 146 as well as by TRR 165 Waves to Weather. M.L. gratefully acknowledges support of the Gutenberg Research College of University Mainz. The authors wish to thank E.~Feireisl (Prague) and A.~Novotn\'y (Toulon) for fruitful discussions.}}

\author{M\' aria Luk\' a\v cov\' a\,-Medvi\softd ov\' a
\and Andreas Sch\"omer
}

\date{\today}

\maketitle

\bigskip

\centerline{Institute of Mathematics, Johannes Gutenberg-University Mainz}
\centerline{Staudingerweg 9, 55128 Mainz, Germany}
\centerline{lukacova@uni-mainz.de, anschoem@uni-mainz.de}

\begin{abstract}
We establish a DMV-strong uniqueness result for the compressible Navier-Stokes system with potential temperature transport.
The concept of generalized,  the so-called dissipative measure-valued (DMV), solutions was proposed in \citep{Lukacova_Schoemer_existence}, where their global-in-time existence was proved.
Here we show that strong solutions are stable in the class of DMV solutions. More precisely, a DMV solution coincides with a strong solution emanating from the same initial data as 
long as the strong solution exists.
\end{abstract}


\section{Introduction}\label{sec_intro}
In meteorological applications the following system of compressible Navier-Stokes equations governing 
the motion of
viscous Newtonian fluid is often used, see, e.g., \citep{Klein1, Klein, Lukacova_Wiebe_1, Lukacova_Wiebe_2},
\begin{align}
    \pt\varrho + \divx(\varrho\fatu) &= 0,\phantom{\divx(\S(\gradx\fatu))}  \label{cequation} \\[2mm]
    \pt(\varrho\fatu) + \divx(\varrho\fatu\otimes\fatu) + \gradx\hspace{0.2mm} \p(\varrho\theta) &= \divx(\S(\gradx\fatu)),\phantom{0}  \label{mequation} \\[2mm]
    \pt(\varrho\theta) + \divx(\varrho\theta\fatu) &= 0, \phantom{\divx(\S(\gradx\fatu))} \label{zequation}
\end{align}
where $\varrho\geq 0$, $\fatu$, and $\theta\geq 0$, denote the \textit{fluid density},  \textit{velocity}, and \textit{potential temperature}, respectively. The \textit{viscous stress tensor} $\S(\gradx\fatu)$ is determined by the stipulation
\begin{align}
    \S(\gradx\fatu) = \mu\!\left(\gradx\fatu + (\gradx\fatu)^T-\frac{2}{d}\,\divx(\fatu)\Id\right)+\lambda\,\divx(\fatu)\,\Id\,, \label{stress_tensor_0}
\end{align}
where the \textit{viscosity constants} $\mu$ and $\lambda$ satisfy
$    \mu > 0 $ and  $\lambda \geq -\frac{2}{d}\,\mu\, .$
The state equation for the \textit{pressure} $p$ reads
\begin{align}
    \p(\varrho\theta) = a(\varrho\theta)^\gamma\,, \qquad a = \const. >0\,, \label{isen_press}
\end{align}
where $\gamma>1$ is the so-called \textit{adiabatic index}.  System \eqref{cequation}--\eqref{zequation}  is solved on $(0,T) \times \Omega$, where $T>0$ is a given time and $\Omega\subset\reals^d$ a bounded domain, $d\in\{2,3\}.$ It is  accompanied with
 the initial data
\begin{align}
    \varrho(0,\cdot) = \varrho_0 \,, \qquad \theta(0,\cdot) = \thetazero \,, \qquad \fatu(0,\cdot) = \fatu_0\,, \label{initial}
\end{align}
and  no-slip boundary conditions
\begin{align}
    \fatu|_{[0,T]\times\po} = \mathbf{0}\,. \label{no_slip}
\end{align}

In the sequel, we shall call system (\ref{cequation})--(\ref{isen_press}) \textit{Navier-Stokes system with potential temperature transport}. For a brief overview of analytical results for this system we refer to  our recent paper \citep{Lukacova_Schoemer_existence}. It is to be pointed out that the existence of global-in-time weak solutions to (\ref{cequation})--(\ref{isen_press})
is available in three space dimensions only for 
$\gamma\geq 9/5$,  see Maltese et al.~\citep[Theorem 1 with $\mathcal{T}(s)=s^\gamma$]{Zatorska}. We note in passing that  a specific choice of the function $\mathcal{T}$ in \citep{Zatorska} yields $s=\theta$ and thus
the Navier-Stokes equations with potential temperature transport. More importantly, physically relevant values of the adiabatic index $\gamma$ lie in the interval $(1,5/3]$ for $d=3$. However, this is \emph{not} the case when the existence of global-in-time weak solutions is available. This drawback motivated our recent paper \citep{Lukacova_Schoemer_existence}, where we have identified a larger class of generalized solutions--\emph{dissipative measure-valued (DMV)
solutions} to the Navier-Stokes system with potential temperature transport. Analyzing the  convergence of a suitable numerical scheme, the mixed finite element--finite volume  method, we have proved global-in-time existence of DMV solutions
for all adiabatic indices $\gamma > 1$ for $d=2,3.$

The goal of the present paper is to show that the strong solutions to the Navier-Stokes system with potential temperature transport are stable in the class of DMV solutions. To this end we 
establish a \textit{DMV-strong uniqueness principle}. This result states that  the DMV and strong solutions emanating from the same initial data coincide. The key concept for the proof of this principle is the \textit{relative energy}: Once a suitable relative energy is identified and the corresponding relative energy inequality is derived, the proof of the DMV-strong uniqueness principle is essentially a consequence of Gronwall's lemma. This strategy for proving DMV/weak-strong uniqueness is not new; see, e.g., \citep{Feireisl_Wiedemann}, where DMV-strong uniqueness is proven for the Navier-Stokes system, and \citep[Chapter 6]{Feireisl_Lukacova_Book}, where DMV-strong uniqueness is proven for the barotropic Euler system, the complete Euler system, and the Navier-Stokes system. However, till now the weak-strong uniqueness principle was not available for the Navier-Stokes equations with potential temperature transport (\ref{cequation})--(\ref{isen_press}). The key difficulty lies in the pressure law that only depends on the total potential temperature $\varrho \theta$, without any independent control of the density $\varrho$. To cure this problem,  we will rewrite the pressure  as a function of the density and total physical entropy. This allows us to separate the effects of the density and potential temperature in the derivation of the relative energy and finally to show the DMV-strong uniqueness principle.

The paper is organized as follows: In Section \ref{sec_meas_sol}, we briefly repeat the relevant notation and our definition of DMV solutions to Navier-Stokes system with potential temperature transport proposed in \citep{Lukacova_Schoemer_existence}. Section \ref{sec_uniqueness} is devoted to the proof of the DMV-strong uniqueness principle.

\section{DMV solutions}\label{sec_meas_sol}
We start by introducing the \textit{pressure potential}
$P:[0,\infty)\to\reals$ as
\begin{align}
    P(z) = \frac{a}{\gamma-1}\,z^\gamma\,. \label{press_pot}
\end{align}
In what follows we write $\Omega_t = (0,t)\times\Omega$ whenever $t>0.$  
If $\mathcal{V}=\{\mathcal{V}_{(t,\fatx)}\}_{(t,\fatx)\,\in\,\Omega_T}$ is a space-time parametrized probability measure  acting on $\reals^{d+2}$, we write
\begin{align*}
    \omean{\mathcal{V}_{(t,\fatx)};g}\equiv\int_{\reals^{d+2}}g\;\dd\mathcal{V}_{(t,\fatx)}\equiv\int_{\reals^{d+2}}g(\bvarrho,\btheta,\bfatu)\;\dd\mathcal{V}_{(t,\fatx)}(\bvarrho,\btheta,\bfatu)
\end{align*}
whenever $g\in C(\reals^{d+2})$. In particular, we tend to write out the function $g$ in terms of the integration variables $(\bvarrho,\btheta,\bfatu)\in\reals\times\reals\times\reals^d\cong\reals^{d+2}$: if, for example, $g(\bvarrho,\btheta,\bfatu)=\bvarrho\,\bfatu$, then we also write
\begin{align*}
    \omean{\mathcal{V}_{(t,\fatx)};\bvarrho\,\bfatu} \qquad \text{instead of} \qquad \omean{\mathcal{V}_{(t,\fatx)};g}\,.
\end{align*}

We recall the definition of dissipative measure-valued solutions to the Navier-Stokes system with potential temperature transport (\ref{cequation})--(\ref{isen_press}) from \citep{Lukacova_Schoemer_existence}.

\begin{definition}[\textbf{DMV solutions}, {\citep[Definition 2.1]{Lukacova_Schoemer_existence}}]\label{def_meas_sol}
A parametrized probability measure $\mathcal{V}=\{\mathcal{V}_{(t,\fatx)}\}_{(t,\fatx)\,\in\,\Omega_T}$ that satisfies
\begin{gather*}
    \mathcal{V}\in\Lp{\infty}{0.2mm}_{\mathrm{weak}^\star}(\Omega_T;\mathcal{P}(\reals^{d+2}))\,\footnotemark, \qquad \reals^{d+2} = \big\{(\bvarrho,\btheta,\bfatu)\,\big|\,\bvarrho,\btheta\in\reals, \bfatu\in\reals^d\big\}\,,
\end{gather*}
and for which there exists a constant $c_\star>0$ such that
\begin{gather*}
    \mathcal{V}_{(t,\fatx)}\big(\{\bvarrho\geq 0\}\cap\{\btheta\geq c_\star\}\big) = 1 \quad \text{for a.a. $(t,\fatx)\in\Omega_T$,}
\end{gather*}
\footnotetext{$\mathcal{P}(\reals^{d+2})$ denotes the space of probability measures on $\reals^{d+2}$.}is called a \textit{dissipative measure-valued (DMV) solution} to the Navier-Stokes system with potential temperature transport (\ref{cequation})--(\ref{isen_press}) with initial and boundary conditions (\ref{initial}) and (\ref{no_slip}) if it satisfies:
\begin{itemize}
    \item{\textbf{energy inequality}
    \begin{align*}
    \fatu_{\mathcal{V}}\equiv\omean{\mathcal{V};\bfatu}\in\Lp{2}{0mm}(0,T;W^{1,2}_0(\Omega)^d)\,, \qquad \left\langle\mathcal{V};\frac{1}{2}\,\bvarrho\,|\bfatu|^2+P(\bvarrho\,\btheta)\right\rangle\in\Lp{1}{0mm}(\Omega_T)\,,
    \end{align*}
    and the integral inequality
    \begin{align}
        &\into\left\langle\mathcal{V}_{(\tau,\,\cdot\,)};\frac{1}{2}\,\bvarrho\,|\bfatu|^2+P(\bvarrho\,\btheta)\right\rangle\dx + \inttauinto \,\mathbb{S}(\gradx\fatu_{\mathcal{V}}):\gradx\fatu_{\mathcal{V}}\;\dxdt \notag \\[2mm]
        &\qquad \qquad \qquad \qquad + \int_{\,\overline{\Omega}}\dd\mathfrak{E}(\tau) + \int_{\,\overline{\Omega}_\tau}\!\dd\mathfrak{D} \leq \into\left[\,\frac{1}{2}\,\varrho_0|\fatu_0|^2+P(\varrho_0\thetazero)\right]\dx \label{energy_inequ}
    \end{align}
    holds for a.a. $\tau\in(0,T)$ with the \textit{energy concentration defect}
    \begin{align*}
    \mathfrak{E}\in\Lp{\infty}{0.2mm}_{\mathrm{weak}^\star}(0,T;\mathcal{M}^+(\overline{\Omega}))
    \end{align*}
    and the \textit{dissipation defect}
    \begin{align*}
    \mathfrak{D}\in\mathcal{M}^+(\,\overline{\Omega}_T)\,;
    \end{align*}}
    \item{\textbf{continuity equation}
    \begin{align*}
        \omean{\mathcal{V};\bvarrho\,}\in C_{\mathrm{weak}}([0,T];\Lp{\gamma}{0.5mm}(\Omega))\,, \quad \omean{\mathcal{V}_{(0,\fatx)};\bvarrho\,} = \varrho_0(\fatx) \;\; \text{for a.a. $\fatx\in\Omega$}
    \end{align*}
    and the integral identity
    \begin{align}
        \left[\,\into\,\omean{\mathcal{V}_{(t,\,\cdot\,)};\bvarrho\,}\,\varphi(t,\cdot)\;\dx\right]_{t\,=\,0}^{t\,=\,\tau} = \inttauinto\Big[\omean{\mathcal{V};\bvarrho\,}\,\pt\varphi + \omean{\mathcal{V};\bvarrho\,\bfatu}\cdot\gradx\varphi\Big]\;\dxdt
    \end{align}
    holds for all $\tau\in[0,T]$ and all $\varphi\in W^{1,\infty}(\Omega_T)$\footnotemark;\footnotetext{Here, the (Lipschitz) continuous representative of $\varphi\in W^{1,\infty}(\Omega_T)$ is meant.}}
    \item{\textbf{momentum equation}
    \begin{align*}
        \omean{\mathcal{V};\bvarrho\,\bfatu}\in C_{\mathrm{weak}}([0,T];\Lp{\frac{2\gamma}{\gamma+1}}{0.5mm}(\Omega)^d)\,, \quad \omean{\mathcal{V}_{(0,\fatx)};\bvarrho\,\bfatu} = \varrho_0(\fatx)\fatu_0(\fatx) \;\; \text{for a.a. $\fatx\in\Omega$}
    \end{align*}
    and the integral identity
    \begin{align}
    \left[\,\into\,\omean{\mathcal{V}_{(t,\,\cdot\,)};\bvarrho\,\bfatu}\cdot\bm{\varphi}(t,\cdot)\;\dx\right]_{t\,=\,0}^{t\,=\,\tau}
    &= \inttauinto\Big[\omean{\mathcal{V};\bvarrho\,\bfatu}\cdot\pt\bm{\varphi} + \omean{\mathcal{V};\bvarrho\,\bfatu\otimes\bfatu+\p(\bvarrho\,\btheta)\Id}:\gradx\bm{\varphi}\Big]\;\dxdt \notag \\[2mm]
    &\phantom{\,=\,}-\inttauinto \,\mathbb{S}(\gradx\fatu_{\mathcal{V}}):\gradx\bm{\varphi}\;\dxdt + \inttauinto\gradx\bm{\varphi}:\dd\bm{\mathfrak{R}}(t)\hspace{0.3mm}\dt
    \end{align}
    holds for all $\tau\in[0,T]$ and all $\bm{\varphi}\in\Cone(\,\overline{\Omega}_T)^d$ satisfying $\bm{\varphi}|_{[0,T]\times\po}=\bm{0}$, where the \textit{Reynolds concentration defect} fulfills
    \begin{gather*}
        \bm{\mathfrak{R}}\in\Lp{\infty}{0.2mm}_{\mathrm{weak}^\star}(0,T;\mathcal{M}(\overline{\Omega})^{d\times d}_{\mathrm{sym},+})\,\footnotemark \\[2mm]
        \text{and} \qquad \underline{d}\hspace{0.3mm}\mathfrak{E}\leq\mathrm{tr}(\bm{\mathfrak{R}})\leq\overline{d}\hspace{0.3mm}\mathfrak{E} \quad \text{for some constants $\overline{d}\geq\underline{d}> 0$;}
    \end{gather*}
    \footnotetext{$\mathcal{M}(\overline{\Omega})^{d\times d}_{\mathrm{sym},+}$ denotes the set of bounded Radon measures defined on $\overline{\Omega}$ and ranging in the set of symmetric positive semi-definite matrices, i.e., $\mathcal{M}(\overline{\Omega})^{d\times d}_{\mathrm{sym},+} = \left\{\mu\in\mathcal{M}(\overline{\Omega})^{d\times d}_{\mathrm{sym}}\,\left|\,\int_{\,\overline{\Omega}}\,\phi(\xi\otimes\xi):\dd\mu \geq 0 \;\text{for all} \; \xi\in\reals^d, \; \phi\in C(\overline{\Omega}), \; \phi\geq 0\right.\right\}$.}}
    \item{\textbf{potential temperature equation}
    \begin{align*}
        \omean{\mathcal{V};\bvarrho\,\btheta\,}\in C_{\mathrm{weak}}([0,T];\Lp{\gamma}{0.5mm}(\Omega))\,, \quad \omean{\mathcal{V}_{(0,\fatx)};\bvarrho\,\btheta\,} = \varrho_0(\fatx)\thetazero(\fatx) \;\; \text{for a.a. $\fatx\in\Omega$}
    \end{align*}
    and the integral identity
    \begin{align}
        \left[\,\into\,\omean{\mathcal{V}_{(t,\,\cdot\,)};\bvarrho\,\btheta\,}\,\varphi(t,\cdot)\;\dx\right]_{t\,=\,0}^{t\,=\,\tau} = \inttauinto\Big[\omean{\mathcal{V};\bvarrho\,\btheta\,}\,\pt\varphi + \omean{\mathcal{V};\bvarrho\,\btheta\,\bfatu}\cdot\gradx\varphi\Big]\;\dxdt
    \end{align}
    holds for all $\tau\in[0,T]$ and all $\varphi\in W^{1,\infty}(\Omega_T)$;}
    \item{\textbf{entropy inequality}
    \begin{align*}
        \omean{\mathcal{V}_{(0,\fatx)};\bvarrho\hspace{0.3mm}\ln(\btheta)} = \varrho_0(\fatx)\ln(\thetazero(\fatx)) \;\;\text{for a.a. $\fatx\in\Omega$}
    \end{align*}
    and for any $\psi\in W^{1,\infty}(\Omega_T)$, $\psi\geq 0$, the integral inequality
    \begin{align}
        \left[\,\into\omean{\mathcal{V}_{(t,\,\cdot\,)};\bvarrho\hspace{0.3mm}\ln(\btheta)}\,\psi(t,\cdot)\;\dx\right]^{t\,=\,\tau}_{t\,=\,0} \geq \inttauinto \big[\omean{\mathcal{V};\bvarrho\hspace{0.3mm}\ln(\btheta)}\,\pt\psi + \omean{\mathcal{V};\bvarrho\hspace{0.3mm}\ln(\btheta)\bfatu}\cdot\gradx\psi\big]\,\dxdt \label{entropy_inequ}
    \end{align}
    is satisfied for a.a. $\tau\in(0,T)$;}
    \item{\textbf{Poincaré's inequality} \\
    there exists a constant $C_P>0$ such that
    \begin{align}
    \inttauinto \,\omean{\mathcal{V};|\bfatu-\bm{U}|^2}\;\dxdt \leq C_P\!\left(\,\inttauinto|\gradx(\fatu_{\mathcal{V}}-\bm{U})|^2\;\dxdt+ \inttau\!\!\int_{\,\overline{\Omega}}\,\dd\mathfrak{E}(t)\hspace{0.3mm}\dt + \int_{\,\overline{\Omega}_\tau}\!\dd\mathfrak{D}\right) \label{poincare_inequ}
    \end{align}
    for a.a. $\tau\in(0,T)$ and all $\bm{U}\in\Lp{2}{0mm}(0,T;W^{1,2}_0(\Omega)^d)$.}
\end{itemize}
\end{definition}


\begin{remark}
As we shall see in the next section, the entropy inequality (\ref{entropy_inequ}) and Poincaré's inequality (\ref{poincare_inequ}) included in the definition of DMV solutions to the Navier-Stokes system with potential temperature transport are fundamental to guarantee DMV-strong uniqueness.
\end{remark}



\section{DMV-strong uniqueness}\label{sec_uniqueness}
The aim of this section is to derive a DMV-strong uniqueness principle for our measure-valued solutions. For this purpose, we rely on the concept of relative energy. 
We introduce the \textit{total (physical) entropy} $S$ as
\begin{align}
    S = S(\varrho,\theta) = \left\{\begin{array}{cl}
        \varrho\ln\!\left((a\theta^\gamma)^\frac{1}{\gamma-1}\right) & \text{if $\theta >0$,} \\[2mm]
        -\infty & \text{if $\theta=0$}
    \end{array}\right. \label{def_S}
\end{align}
and realize that the pressure $p=a(\varrho\theta)^\gamma$ can be rewritten with respect to $\varrho$, $S$ as
\begin{align}
    \p(\varrho,S) = \left\{\begin{array}{cl}
        \varrho^\gamma \exp\left((\gamma-1)\,\dfrac{S}{\varrho}\right) & \text{if $\varrho>0$ and $S\in\reals$,} \\[4mm]
        0 & \text{if $\varrho=0$ and $S\leq 0$, or $S=-\infty$,} \\[2mm]
        \infty & \text{if $\varrho=0$ and $S>0$.}
    \end{array}\right. \label{press_rho_S}
\end{align}
We proceed by defining the relative energy between a triplet of arbitrary  functions $(\varrho,\theta,\fatu)$ belonging to a regularity class
\begin{align}
    \varrho,\theta\in \Ck{1}(\overline{\Omega}_T)\,, \quad \varrho,\theta>0\,, \quad \fatu\in\Ck{1}(\overline{\Omega}_T) \cap 
    L^2(0,T; W^{2,\infty}(\Omega))\,, \quad \fatu|_{[0,T]\times\po}=\bm{0}, \label{adm_test_func_rel_energy_inequ}
\end{align}
and a DMV solution $\mathcal{V}$ to the Navier-Stokes system with potential temperature transport (\ref{cequation})--(\ref{isen_press})  as
\begin{align}
    E(\mathcal{V}\hspace{0.3mm}|\hspace{0.3mm}\varrho,\theta,\fatu) = \Big\langle\mathcal{V};\frac{1}{2}\,\bvarrho\,|\bfatu-\fatu|^2 + P(\bvarrho,\bS) - \frac{\partial P(\varrho,S)}{\partial \varrho}\,(\bvarrho-\varrho) - \frac{\partial P(\varrho,S)}{\partial S}\,(\bS-S) - P(\varrho,S) \Big\rangle\,, \label{def_rel_energy}
\end{align}
where $P(\varrho,S)=\frac{1}{\gamma-1}\,\p(\varrho,S)$ is the pressure potential expressed in terms of $\varrho$ and $S$, $S=S(\varrho,\theta)$, and $\bS=S(\bvarrho,\btheta)$. We note that the relative energy defined in (\ref{def_rel_energy}) is the generalization of the relative energy used in \citep[Formula (4.59)]{Feireisl_Lukacova_Book} in the context of weak solutions. The corresponding relative energy inequality reads as follows.

\begin{lemma}[\textbf{Relative energy inequality}]\label{lem_rel_energy_inequ}
Let $(\varrho,\theta,\fatu)$ be a triplet of test functions, cf. $\mathrm{(\ref{adm_test_func_rel_energy_inequ})}$, and $\mathcal{V}$ a DMV solution to $\mathrm{(\ref{cequation})}$--$\mathrm{(\ref{isen_press})}$ in the sense of Definition $\mathrm{\ref{def_meas_sol}}$. Then the relative energy defined in $\mathrm{(\ref{def_rel_energy})}$ satisfies the inequality
\begin{align}
    &\left[\,\into E(\mathcal{V}_{(t,\,\cdot\,)}\hspace{0.3mm}|\hspace{0.3mm}\varrho,\theta,\fatu)\;\dx\right]_{t\,=\,0}^{t\,=\,\tau} + \int_{\,\overline{\Omega}}\,\dd\mathfrak{E}(\tau) + \int_{\,\overline{\Omega}_\tau}\dd\mathfrak{D} + \inttauinto\mathbb{S}(\gradx(\fatu_\mathcal{V}-\fatu)):\gradx(\fatu_\mathcal{V}-\fatu)\;\dxdt \notag \\[2mm]
    &\quad \leq -\inttauinto \big\langle\mathcal{V};\bvarrho\,(\bfatu-\fatu)^T\cdot\gradx\fatu\cdot(\bfatu-\fatu)\big\rangle\,\dxdt \notag \\[2mm]
    &\quad \phantom{\,\leq\,} -\inttauinto \left\langle\mathcal{V};\p(\bvarrho,\bS) - \frac{\partial \p(\varrho,S)}{\partial \varrho}\,(\bvarrho-\varrho) - \frac{\partial \p(\varrho,S)}{\partial S}\,(\bS-S) - \p(\varrho,S)\right\rangle\divx(\fatu)\;\dxdt \notag \\[2mm]
    &\quad \phantom{\,\leq\,} +\inttauinto \left\langle\mathcal{V};\frac{\bvarrho}{\varrho}\,(\fatu-\bfatu)\right\rangle\cdot\big[\varrho\pt\fatu+\varrho\gradx\fatu\cdot\fatu+\gradx \p(\varrho,S) - \divx(\mathbb{S}(\gradx\fatu))\big]\,\dxdt \notag \\[2mm]
    &\quad \phantom{\,\leq\,} + \inttauinto \left\langle\mathcal{V};(\varrho-\bvarrho\,)\,\frac{1}{\varrho}\,\frac{\partial \p(\varrho,S)}{\partial\varrho}\right\rangle\big[\pt\varrho+\divx(\varrho\fatu)\big]\,\dxdt \notag \\[2mm]
    &\quad \phantom{\,\leq\,} + \inttauinto \left\langle\mathcal{V};(\varrho-\bvarrho\,)\,\frac{1}{\varrho}\,\frac{\partial \p(\varrho,S)}{\partial S}\right\rangle\big[\pt S+\divx(S\fatu)\big]\,\dxdt \notag \\[2mm]
    &\quad \phantom{\,\leq\,} + \inttauinto \left\langle\mathcal{V};\frac{\bvarrho}{\varrho}\,S-\bS\right\rangle\left[\pt\vartheta+\fatu\cdot\gradx\vartheta+\frac{\partial \p(\varrho,S)}{\partial S}\,\divx(\fatu)\right]\dxdt \notag \\[2mm]
    &\quad \phantom{\,\leq\,} +\inttauinto \left\langle\mathcal{V};\left(\frac{\bvarrho}{\varrho}\,S-\bS\right)(\bfatu-\fatu)\right\rangle\cdot\gradx\vartheta\;\dxdt - \inttauinto\gradx\fatu:\dd\bm{\mathfrak{R}}(t)\hspace{0.3mm}\dt \notag \\[2mm]
    &\quad \phantom{\,\leq\,} +\inttauinto \left\langle\mathcal{V};(\bvarrho-\varrho)\,\frac{1}{\varrho}\,\divx(\mathbb{S}(\gradx\fatu))\cdot(\fatu-\bfatu)\right\rangle\dxdt \label{rel_energy_inequ}
\end{align}
for a.a. $\tau\in(0,T)$. Here,
\begin{align*}
    \vartheta = \frac{1}{\gamma-1}\,\frac{\partial \p(\varrho,S)}{\partial S}
\end{align*}
denotes the absolute temperature.
\end{lemma}

\begin{proof}
Using Gauss's theorem we easily verify that
\begin{align*}
    -\inttauinto \mathbb{S}(\gradx\fatu):\gradx(\fatu_\mathcal{V}-\fatu)\;\dxdt &= - \inttauinto \left\langle\mathcal{V};\frac{\bvarrho}{\varrho}\,(\fatu-\bfatu)\right\rangle\cdot\divx(\mathbb{S}(\gradx\fatu))\;\dxdt \notag \\[2mm]
    &\phantom{\,=\,} + \inttauinto \left\langle\mathcal{V};(\bvarrho-\varrho)\,\frac{1}{\varrho}\,\divx(\mathbb{S}(\gradx\fatu))\cdot(\fatu-\bfatu)\right\rangle\dxdt\,.
\end{align*}
Thus, to prove inequality (\ref{rel_energy_inequ}), it suffices to realize that
\begin{align*}
    \left[\,\into E(\mathcal{V}_{(t,\,\cdot\,)}\hspace{0.3mm}|\hspace{0.3mm}\varrho,\theta,\fatu)\;\dx\right]_{t\,=\,0}^{t\,=\,\tau} &= \left[\,\into\left\langle\mathcal{V}_{(t,\,\cdot\,)};\frac{1}{2}\,\bvarrho\,|\bfatu|^2 + P(\bvarrho,\bS)\right\rangle\dx\right]_{t\,=\,0}^{t\,=\tau} + \left[\,\into \p(\varrho,S)\;\dx\right]_{t\,=\,0}^{t\,=\tau} \\[2mm]
    &\phantom{\,=\,}  - \left[\,\into\omean{\mathcal{V}_{(t,\,\cdot\,)};\bS}\,\frac{\partial P(\varrho,S)}{\partial S}\;\dx\right]_{t\,=\,0}^{t\,=\tau} - \left[\,\into\omean{\mathcal{V}_{(t,\,\cdot\,)};\bvarrho\,\bfatu}\cdot\fatu\;\dx\right]_{t\,=\,0}^{t\,=\tau} \\[2mm]
    &\phantom{\,=\,} + \left[\,\into\omean{\mathcal{V}_{(t,\,\cdot\,)};\bvarrho\,}\left(\frac{1}{2}\,|\fatu|^2-\frac{\partial P(\varrho,S)}{\partial\varrho}\right)\dx\right]_{t\,=\,0}^{t\,=\tau}
\end{align*}
and utilize (\ref{energy_inequ})--(\ref{entropy_inequ}) to rewrite the terms on the right-hand side. We omit the necessary computations since they are straightforward and very similar to those leading to \citep[ (4.66)]{Feireisl_Lukacova_Book}.
\end{proof}

From the relative energy inequality we can deduce DMV-strong uniqueness.

\begin{theorem}[DMV-strong uniqueness]\label{cor_dmv_strong_uniqueness}
Let $\gamma>1$,  $\Omega\subset\reals^d$, $d\in\{2,3\}$, be a bounded Lipschitz-continuous domain. Further, let $T^* > 0 $ and $(\varrho,\theta,\fatu)$ be a strong solution to system $\mathrm{(\ref{cequation})}$--$\mathrm{(\ref{isen_press})}$ on $\Omega_{T^*}$ belonging to the regularity class $\mathrm{(\ref{adm_test_func_rel_energy_inequ})}.$   Let
$\mathcal{V}$ be a DMV solution in the sense of Definition $\mathrm{\ref{def_meas_sol}}$ emanating from the same initial data. Then
\begin{align*}
    \mathfrak{E}=0\,, \quad \mathfrak{D}=0\,, \quad \bm{\mathfrak{R}}=\bm{0}\,,
\end{align*}
and the DMV and strong solutions coincide on $[0,T^*]$, i.e.
$$
\mathcal{V}_{(t,\fatx)}=\delta_{(\varrho(t,\fatx),\theta(t,\fatx),\fatu(t,\fatx))} \quad \text{for a.a. } (t,\fatx)\in\Omega_{T^*}.
$$
\end{theorem}

\begin{proof}
Plugging the strong solution $(\varrho,\theta,\fatu)$ into the relative energy inequality (\ref{rel_energy_inequ}), we obtain
\begin{align}
    &\left[\,\into E(\mathcal{V}_{(t,\,\cdot\,)}\hspace{0.3mm}|\hspace{0.3mm}\varrho,\theta,\fatu)\;\dx\right]_{t\,=\,0}^{t\,=\,\tau} + \int_{\,\overline{\Omega}}\,\dd\mathfrak{E}(\tau) + \int_{\,\overline{\Omega}_\tau}\dd\mathfrak{D} + \inttauinto\mathbb{S}(\gradx(\fatu_\mathcal{V}-\fatu)):\gradx(\fatu_\mathcal{V}-\fatu)\;\dxdt \notag \\[2mm]
    &\quad \leq -\inttauinto \big\langle\mathcal{V};\bvarrho\,(\bfatu-\fatu)^T\cdot\gradx\fatu\cdot(\bfatu-\fatu)\big\rangle\,\dxdt \notag \\[2mm]
    &\quad \phantom{\,\leq\,} -\inttauinto \left\langle\mathcal{V};\p(\bvarrho,\bS) - \frac{\partial \p(\varrho,S)}{\partial \varrho}\,(\bvarrho-\varrho) - \frac{\partial \p(\varrho,S)}{\partial S}\,(\bS-S) - \p(\varrho,S)\right\rangle\divx(\fatu)\;\dxdt \notag \\[2mm]
    &\quad \phantom{\,\leq\,} +\inttauinto \left\langle\mathcal{V};\left(\frac{\bvarrho}{\varrho}\,S-\bS\right)(\bfatu-\fatu)\right\rangle\cdot\gradx\vartheta\;\dxdt - \inttauinto\gradx\fatu:\dd\bm{\mathfrak{R}}(t)\hspace{0.3mm}\dt \notag \\[2mm]
    &\quad \phantom{\,\leq\,} +\inttauinto \left\langle\mathcal{V};(\bvarrho-\varrho)\,\frac{1}{\varrho}\,\divx(\mathbb{S}(\gradx\fatu))\cdot(\fatu-\bfatu)\right\rangle\dxdt \notag \\[2mm]
    &\quad \lesssim \inttau\left[\,\into E(\mathcal{V}\hspace{0.3mm}|\hspace{0.3mm}\varrho,\theta,\fatu)\;\dx + \int_{\,\overline{\Omega}}\,\dd\mathfrak{E}(t)\right]\dt +\inttauinto \left\langle\mathcal{V};\left(\frac{\bvarrho}{\varrho}\,S-\bS\right)(\bfatu-\fatu)\right\rangle\cdot\gradx\vartheta\;\dxdt \notag \\[2mm]
    &\quad \phantom{\,\leq\,} +\inttauinto \left\langle\mathcal{V};(\bvarrho-\varrho)\,\frac{1}{\varrho}\,\divx(\mathbb{S}(\gradx\fatu))\cdot(\fatu-\bfatu)\right\rangle\dxdt
    \label{dmv_strong_uniqueness_1}
\end{align}
for a.a. $\tau\in(0,T^*)$. To handle the last two integrals, we first observe that
\begin{align}
    \inttauinto\mathbb{S}(\gradx(\fatu_\mathcal{V}-\fatu)):\gradx(\fatu_\mathcal{V}-\fatu)\;\dxdt &= \inttauinto\big[\mu|\gradx(\fatu_\mathcal{V}-\fatu)|^2+\nu|\divx(\fatu_\mathcal{V}-\fatu)|^2\big]\,\dxdt \notag \\[2mm]
    &\geq \mu\inttauinto |\gradx(\fatu_\mathcal{V}-\fatu)|^2\;\dxdt\,. \label{dmv_strong_uniqueness_2}
\end{align}
Next, we set
\begin{align*}
    (\underline{\varrho},\overline{\varrho},\underline{\theta},\overline{\theta}) = \left(\inf_{(t,\fatx)\,\in\,\Omega_{T^*}}\{\varrho(t,\fatx)\}, \sup_{(t,\fatx)\,\in\,\Omega_{T^*}}\{\varrho(t,\fatx)\},\inf_{(t,\fatx)\,\in\,\Omega_{T^*}}\{\theta(t,\fatx)\}, \sup_{(t,\fatx)\,\in\,\Omega_{T^*}}\{\theta(t,\fatx)\}\right)
\end{align*}
and apply Lemma \ref{aux_lem_rel_energy} to find constants $c_1,c_2,c_3>0$ that only depend on $\underline{\varrho}$, $\overline{\varrho}$, $\underline{\theta}$, $\overline{\theta}$, $c_\star$, and $\gamma$, and corresponding sets
\begin{align*}
    \mathcal{R} &= \left\{(\bvarrho,\btheta,\bfatu)\in\reals^{d+2}\left|\,c_1\underline{\varrho}\leq\bvarrho\leq c_2\overline{\varrho}\,, \; c_\star\leq\btheta \leq c_3\overline{\theta} \right.\right\}, \\[2mm]
    \mathcal{S} &= \left\{(\bvarrho,\btheta,\bfatu)\in\reals^{d+2}\left|\,\bvarrho\geq 0\,, \; \btheta\geq c_\star \right.\right\}\big\backslash \mathcal{R}
\end{align*}
such that
\begin{align}
    \inttauinto E(\mathcal{V}\hspace{0.3mm}|\hspace{0.3mm}\varrho,\theta,\fatu)\;\dxdt &\gtrsim \inttauinto \big\langle\mathcal{V}; \mathds{1}_\mathcal{R}(\bvarrho,\btheta,\bfatu)\big(|\bfatu-\fatu|^2 + |\bvarrho-\varrho|^2 + |\bS-S|^2\big) \big\rangle\,\dxdt \notag \\[2mm]
    &\phantom{\,\gtrsim\,} +\inttauinto \big\langle\mathcal{V}; \mathds{1}_\mathcal{S}(\bvarrho,\btheta,\bfatu)\big(1+\bvarrho\,|\bfatu-\fatu|^2 + (\bvarrho\btheta)^\gamma\big)\big\rangle\,\dxdt\,. \label{help}
\end{align}
Seeing that
\begin{align*}
    \left|\left(\frac{\bvarrho}{\varrho}\,S-\bS\right)(\bfatu-\fatu)\right| &\lesssim \big|(\bvarrho S -\bS\varrho)(\bfatu-\fatu)\big| \lesssim \big|S(\bvarrho-\varrho)(\bfatu-\fatu)\big| + \big|\varrho(S-\bS)(\bfatu-\fatu)\big| \notag \\[2mm]
    &\lesssim |\bfatu-\fatu|^2 + |\bvarrho-\varrho|^2 + |\bS-S|^2
\end{align*}
as well as
\begin{align*}
    &\left|\mathds{1}_\mathcal{S}(\bvarrho,\btheta,\bfatu)\left(\frac{\bvarrho}{\varrho}\,S-\bS\right)(\bfatu-\fatu)\right| \notag \\[2mm]
    &\qquad \lesssim \mathds{1}_\mathcal{S}(\bvarrho,\btheta,\bfatu)\left(\bvarrho\,|\bfatu-\fatu| + \bS\,|\bfatu-\fatu|\right) \lesssim \mathds{1}_\mathcal{S}(\bvarrho,\btheta,\bfatu)\left(\bvarrho\,|\bfatu-\fatu| + \bvarrho\btheta^{\,1/2}\,|\bfatu-\fatu|\right) \notag \\[2mm]
    &\qquad \lesssim \mathds{1}_\mathcal{S}(\bvarrho,\btheta,\bfatu)\left(\bvarrho + \bvarrho\btheta + \bvarrho\,|\bfatu-\fatu|^2\right) \lesssim \mathds{1}_\mathcal{S}(\bvarrho,\btheta,\bfatu)\left(1 + \bvarrho\,|\bfatu-\fatu|^2 + (\bvarrho\btheta)^\gamma\right),
\end{align*}
we may use (\ref{help}) to deduce
\begin{align}
    \left|\,\inttauinto \left\langle\mathcal{V};\left(\frac{\bvarrho}{\varrho}\,S-\bS\right)(\bfatu-\fatu)\right\rangle\cdot\gradx\vartheta\;\dxdt\,\right| \lesssim \inttauinto E(\mathcal{V}\hspace{0.3mm}|\hspace{0.3mm}\varrho,\theta,\fatu)\;\dxdt\,. \label{dmv_strong_uniqueness_3}
\end{align}
We proceed by observing that
\begin{align*}
    \big|(\bvarrho-\varrho)(\fatu-\bfatu)\big| \lesssim |\bvarrho-\varrho|^2 + |\bfatu-\fatu|^2
\end{align*}
and
\begin{align*}
    \big|\mathds{1}_\mathcal{S}(\bvarrho,\btheta,\bfatu)(\bvarrho-\varrho)(\fatu-\bfatu)\big| &\lesssim \mathds{1}_\mathcal{S}(\bvarrho,\btheta,\bfatu)(\bvarrho+1)|\bfatu-\fatu| \notag \\[2mm]
    &\lesssim \mathds{1}_\mathcal{S}(\bvarrho,\btheta,\bfatu)\left(\bvarrho + \bvarrho\,|\bfatu-\fatu|^2 + \alpha|\bfatu-\fatu|^2 + \alpha^{-1}\right) \notag \\[2mm]
    &\lesssim \mathds{1}_\mathcal{S}(\bvarrho,\btheta,\bfatu)\left(1+(\bvarrho\btheta)^\gamma + \bvarrho\,|\bfatu-\fatu|^2 + \alpha|\bfatu-\fatu|^2 + \alpha^{-1}\right)
\end{align*}
for all $\alpha>0$, where here and in the sequel the constant hidden in "$\lesssim$" does not depend on $\alpha$. Together with (\ref{help}) and Poincaré's inequality (\ref{poincare_inequ}), these observations yield
\begin{align}
    &\left|\,\inttauinto \left\langle\mathcal{V};(\bvarrho-\varrho)\,\frac{1}{\varrho}\,\divx(\mathbb{S}(\gradx\fatu))\cdot(\fatu-\bfatu)\right\rangle\dxdt\,\right| \notag \\[2mm]
    &\lesssim (1+\alpha^{-1}) \inttauinto E(\mathcal{V}\hspace{0.3mm}|\hspace{0.3mm}\varrho,\theta,\fatu)\;\dxdt + \alpha\left(\,\inttauinto|\gradx(\fatu_{\mathcal{V}}-\bm{u})|^2\;\dxdt+ \inttau\!\!\int_{\,\overline{\Omega}}\,\dd\mathfrak{E}(t)\hspace{0.3mm}\dt + \int_{\,\overline{\Omega}_\tau}\!\dd\mathfrak{D}\right)\,. \label{dmv_strong_uniqueness_4}
\end{align}
Finally, combining (\ref{dmv_strong_uniqueness_1}), (\ref{dmv_strong_uniqueness_2}), (\ref{dmv_strong_uniqueness_3}), and (\ref{dmv_strong_uniqueness_4}), we arrive at
\begin{align*}
    &\left[\,\into E(\mathcal{V}_{(t,\,\cdot\,)}\hspace{0.3mm}|\hspace{0.3mm}\varrho,\theta,\fatu)\;\dx\right]_{t\,=\,0}^{t\,=\,\tau} + \int_{\,\overline{\Omega}}\,\dd\mathfrak{E}(\tau) + \int_{\,\overline{\Omega}_\tau}\dd\mathfrak{D} + \mu\inttauinto |\gradx(\fatu_\mathcal{V}-\fatu)|^2\;\dxdt \notag \\[2mm]
    &\quad \lesssim (1+\alpha^{-1}) \inttauinto E(\mathcal{V}\hspace{0.3mm}|\hspace{0.3mm}\varrho,\theta,\fatu)\;\dxdt + (1+\alpha)\inttau\!\!\int_{\,\overline{\Omega}}\,\dd\mathfrak{E}(t)\hspace{0.3mm}\dt \notag \\[2mm]
    &\quad \phantom{\,\lesssim\,} +\alpha\left(\,\inttauinto|\gradx(\fatu_{\mathcal{V}}-\bm{u})|^2\;\dxdt + \int_{\,\overline{\Omega}_\tau}\!\dd\mathfrak{D}\right)
\end{align*}
for a.a. $\tau\in(0,T^*)$ and all $\alpha>0$. In particular, there exists a constant $C>0$ such that
\begin{align*}
    &\left[\,\into E(\mathcal{V}_{(t,\,\cdot\,)}\hspace{0.3mm}|\hspace{0.3mm}\varrho,\theta,\fatu)\;\dx\right]_{t\,=\,0}^{t\,=\,\tau} + \int_{\,\overline{\Omega}}\,\dd\mathfrak{E}(\tau) + \int_{\,\overline{\Omega}_\tau}\dd\mathfrak{D} \notag \\[2mm]
    &\qquad\qquad\qquad \leq C\left(\,\inttauinto E(\mathcal{V}\hspace{0.3mm}|\hspace{0.3mm}\varrho,\theta,\fatu)\;\dxdt + \inttau\!\!\int_{\,\overline{\Omega}}\,\dd\mathfrak{E}(t)\hspace{0.3mm}\dt + \inttau\!\!\int_{\,\overline{\Omega}_t}\dd\mathfrak{D}\hspace{0.3mm}\dt\right)
\end{align*}
for a.a. $\tau\in(0,T^*)$. Consequently, the desired result follows from Gronwall's lemma.
\end{proof}

\section{Conclusions}\label{sec_conlusions}
In the present paper, we proved the DMV-strong uniqueness principle for the Navier-Stokes system with potential temperature transport (\ref{cequation})--(\ref{isen_press}). In fact, this result shows that strong solutions are stable in the class of DMV solutions introduced in \citep{Lukacova_Schoemer_existence}. We have  derived the relative energy by  taking the total physical entropy into account.  More precisely, the pressure was rewritten as a function of the density and entropy, instead of the  total potential temperature only. Moreover, we also  require
the entropy inequality (\ref{entropy_inequ}) that is included in our definition of DMV solutions. The importance of Poincaré's inequality (\ref{poincare_inequ}) became clear during the proof of DMV-strong uniqueness: It allowed us to rewrite viscosity terms in such a way that Gronwall's lemma was applicable. Finally, the
DMV-strong uniqueness result follows  by applying
 Gronwall's lemma.

The DMV-strong uniqueness principle was used in our recent work \citep{Lukacova_Schoemer_existence}. In Theorem 6.1  we relied on this result to prove the strong convergence of the numerical solutions of our mixed finite element--finite volume scheme \citep[Definition 3.2]{Lukacova_Schoemer_existence} to the classical solution of the system  as long as  the latter exists.

\bibliographystyle{plain}
\bibliography{Bibliography}

\begin{thebibliography}{1}

\bibitem{Lukacova_Wiebe_2}
A.~Chertock, A.~Kurganov, M.~Lukáčová-Medvid’ová, P.~Spichtinger, and
  B.~Wiebe.
\newblock Stochastic {G}alerkin method for cloud simulation.
\newblock {\em Math. Clim. Weather Forecast.}, 5(1):65--106, 2019.

\bibitem{Feireisl_Wiedemann}
E.~Feireisl, P.~Gwiazda, A.~Świerczewska Gwiazda, and E.~Wiedemann.
\newblock Dissipative measure-valued solutions to the compressible
  {N}avier-{S}tokes system.
\newblock {\em Calc. Var. Partial Differential Equations}, 55(141), 2016.

\bibitem{Klein}
E.~Feireisl, R.~Klein, A.~Novotn\'{y}, and E.~Zatorska.
\newblock On singular limits arising in the scale analysis of stratified fluid
  flows.
\newblock {\em Math. Models Methods Appl. Sci.}, 26(3):419--443, 2016.

\bibitem{Feireisl_Lukacova_Book}
E.~Feireisl, M.~Lukáčová-Medvid’ová, H.~Mizerová, and B.~She.
\newblock {\em Numerical {A}nalysis of {C}ompressible {F}luid {F}lows}, volume
  20 of \textit{MS}\&\textit{A}.
\newblock Springer International Publishing, 2021.

\bibitem{Klein1}
R.~Klein.
\newblock An applied mathematical view of meteorological modelling.
\newblock In {\em Applied mathematics entering the 21st century}, pages
  227--269. SIAM, Philadelphia, PA, 2004.

\bibitem{Lukacova_Wiebe_1}
M.~Luk{\'a}{\v{c}}ov{\'a}-Medvid'ov{\'a}, J.~Rosemeier, P.~Spichtinger, and
  B.~Wiebe.
\newblock I{M}{E}{X} {F}inite {V}olume {M}ethods for {C}loud {S}imulation.
\newblock In {\em Finite Volumes for Complex Applications VIII - Hyperbolic,
  Elliptic and Parabolic Problems}, pages 179--187, Cham, 2017. Springer
  International Publishing.

\bibitem{Lukacova_Schoemer_existence}
M.~Lukáčová-Medvid'ová and A.~Schömer.
\newblock Existence of dissipative solutions to the compressible
  {N}avier-{S}tokes system with potential temperature transport.
\newblock {\tt arXiv:2106.12435 [math.NA]}, 2021.

\bibitem{Zatorska}
D.~Maltese, M.~Michálek, P.B. Mucha, A.~Novotný, M.~Pokorný, and
  E.~Zatorska.
\newblock Existence of weak solutions for compressible {N}avier-{S}tokes
  equations with entropy transport.
\newblock {\em J. Differential Equations}, 261(8):4448--4485, 2016.

\end{thebibliography}

\appendix
\section{Appendix}\label{sec_appendix}
\subsection{An auxiliary result concerning the relative energy}
Here, we prove the auxiliary result used in the proof of DMV-strong uniqueness.

\begin{lemma}\label{aux_lem_rel_energy}
Let $\bvarrho\geq 0$, $\btheta\geq c_\star >0$, $0<\underline{\varrho}\leq\varrho\leq\overline{\varrho}$, $0<\underline{\theta}\leq\theta\leq\overline{\theta}$, and $\gamma >1$. Then there exist constants $c_1,c_2,c_3,c_4>0$ that only depend on $\underline{\varrho},\overline{\varrho},\underline{\theta},\overline{\theta},c_\star$, and $\gamma$, and corresponding sets
\begin{align*}
    \mathcal{R} &= \left\{(\bvarrho,\btheta)\in\reals^{2}\left|\,c_1\underline{\varrho}\leq \bvarrho \leq c_2\overline{\varrho}\,, \; c_\star\leq\btheta \leq c_3\overline{\theta} \right.\right\}, \\[2mm]
    \mathcal{S} &= \left\{(\bvarrho,\btheta)\in\reals^{2}\left|\,\bvarrho\geq 0\,, \; \btheta\geq c_\star \right.\right\}\big\backslash \mathcal{R}
\end{align*}
such that
\begin{align}
    F(\bvarrho,\bS\hspace{0.3mm}|\hspace{0.3mm}\varrho,S) &\equiv P(\bvarrho,\bS) - \frac{\partial P(\varrho,S)}{\partial \varrho}(\bvarrho-\varrho) - \frac{\partial P(\varrho,S)}{\partial S}(\bS-S) - P(\varrho,S) \notag \\[2mm]
    &\geq c_4 \left[\mathds{1}_\mathcal{S}(\bvarrho,\btheta)\big(|\bvarrho-\varrho|^2 + |\bS-S|^2\big) + \mathds{1}_\mathcal{R}(\bvarrho,\btheta)\big(1+(\bvarrho\btheta)^\gamma\big)\right],
\end{align}
where $P(\varrho,S)=\frac{1}{\gamma-1}\,\p(\varrho,S)$ with $p$ from $\mathrm{(\ref{press_rho_S})}$, $S=S(\varrho,\theta)$ is defined in $\mathrm{(\ref{def_S})}$, and $\bS=S(\bvarrho,\btheta)$.
\end{lemma}

\begin{proof}
To begin with, let $0<c_1\leq c_2$, and $c_3\geq c_\star/\,\overline{\theta}$ be arbitrary numbers. Further, let $\mathcal{R}$, $\mathcal{S}$ be defined as described in the lemma. We decompose $\mathcal{S}$ into the sets
\begin{align*}
    \mathcal{S}^+ = \left\{(\bvarrho,\btheta)\in\mathcal{S}\left|\,\bvarrho < c_1\underline{\varrho} \right.\right\}\,, \quad  \mathcal{S}^- = \left\{(\bvarrho,\btheta)\in\mathcal{S}\left|\,\bvarrho > c_2\overline{\varrho} \right.\right\}\,, \quad \mathcal{S}^0 = \mathcal{S}\hspace{0.3mm}\backslash(\mathcal{S}^+\cup\mathcal{S}^-)
\end{align*}
and observe that
\begin{align*}
    F(\bvarrho,\bS\hspace{0.3mm}|\hspace{0.3mm}\varrho,S) &= a\left((\varrho\theta)^\gamma - \frac{\gamma}{\gamma-1}\,\varrho^{\gamma-1}\theta^\gamma\,\bvarrho\,\big(1-\ln(\theta)+\ln(\btheta)\big)+\frac{1}{\gamma-1}\,(\bvarrho\btheta)^\gamma\right) \notag \\[2mm]
    &\geq a\left((\varrho\theta)^\gamma - \frac{\gamma}{\gamma-1}\,\varrho^{\gamma-1}\theta^\gamma\,\bvarrho\,\big(1+|\ln(\theta)|+\btheta^{\,1/2}\big)+\frac{1}{\gamma-1}\,(\bvarrho\btheta)^\gamma\right)
\end{align*}
wherefore
\begin{align*}
    &\mathds{1}_{\mathcal{S}^-}(\bvarrho,\btheta) F(\bvarrho,\bS\hspace{0.3mm}|\hspace{0.3mm}\varrho,S) \notag \\[2mm]
    &\quad \geq a(\underline{\varrho}\underline{\theta})^\gamma-\frac{ac_1\gamma}{\gamma-1}\left(1+\max\left\{|\ln(\underline{\theta})|,|\ln(\overline{\theta})|\right\}\right)(\overline{\varrho}\overline{\theta})^\gamma - c_1^{2\gamma/(2\gamma-1)}\,\frac{a(2\gamma-1)}{2(\gamma-1)}\overline{\varrho}^{\,\gamma-1+\gamma/(2\gamma-1)}\overline{\theta}^{\,\gamma} \notag \\[2mm]
    &\quad \phantom{\,=\,} + \frac{a}{\gamma-1}\left(1- \frac{c_1^{2\gamma}}{2}\,\overline{\varrho}^{\,\gamma-1}\overline{\theta}^{\,\gamma}\right)(\bvarrho\btheta)^\gamma\,, \\[6mm]
    &\mathds{1}_{\mathcal{S}^+}(\bvarrho,\btheta) F(\bvarrho,\bS\hspace{0.3mm}|\hspace{0.3mm}\varrho,S) \notag \\[2mm]
    &\quad \geq a(\underline{\varrho}\underline{\theta})^\gamma + \frac{a}{\gamma-1}\left(1-\gamma c_2^{1-\gamma}\left(\frac{\overline{\theta}}{c_\star}\right)^{\!\!\gamma}\left(1+\max\left\{|\ln(\underline{\theta})|,|\ln(\overline{\theta})|\right\}+c_\star^{1/2}\right)\right)(\bvarrho\btheta)^\gamma\,, \\[6mm]
    &\mathds{1}_{\mathcal{S}^0}(\bvarrho,\btheta) F(\bvarrho,\bS\hspace{0.3mm}|\hspace{0.3mm}\varrho,S) \notag \\[2mm]
    &\quad \geq a(\underline{\varrho}\underline{\theta})^\gamma + \frac{a}{\gamma-1}\left(1-\gamma c_3^{-\gamma}\left(\frac{\overline{\varrho}}{c_2\underline{\varrho}}\right)^{\!\!\gamma-1}\left(1+\max\left\{|\ln(\underline{\theta})|,|\ln(\overline{\theta})|\right\}+(c_3\overline{\theta})^{1/2}\right)\right)(\bvarrho\btheta)^\gamma\,.
\end{align*}
Here, the first inequality is obtained using Young's inequality. Together, the above observations show that we can specify $c_1,c_2,c_3$ in dependence of $\underline{\varrho},\overline{\varrho},\underline{\theta},\overline{\theta},c_\star,\gamma$ such that
\begin{align*}
    \mathds{1}_{\mathcal{S}}(\bvarrho,\btheta) F(\bvarrho,\bS\hspace{0.3mm}|\hspace{0.3mm}\varrho,S) \geq c_{4,1}\mathds{1}_{\mathcal{S}}(\bvarrho,\btheta)\big(1+(\bvarrho\btheta)^\gamma\big)\,,
\end{align*}
where $c_{4,1}>0$ solely depends on $\underline{\varrho},\overline{\varrho},\underline{\theta},\overline{\theta},c_\star,\gamma$. Having fixed $c_1,c_2,c_3$ as described above, it remains to show that
\begin{align*}
    \mathds{1}_{\mathcal{R}}(\bvarrho,\btheta) F(\bvarrho,\bS\hspace{0.3mm}|\hspace{0.3mm}\varrho,S) \geq c_{4,2}\mathds{1}_{\mathcal{R}}(\bvarrho,\btheta)\big(|\bvarrho-\varrho|^2+|\bS-S|^2\big)\,,
\end{align*}
where $c_{4,2}>0$ only depends on $\underline{\varrho},\overline{\varrho},\underline{\theta},\overline{\theta},c_\star,\gamma$. However, this inequality is a direct consequence of the fact that $P=P(\varrho,S)$ is strongly convex on every compact convex subset of $(0,\infty)\times\reals$ which, in turn, follows from the positive definiteness of the Hessian of $P$ on $(0,\infty)\times\reals$.
\end{proof}

\end{document}